\documentclass[12pt,reqno]{amsart}
\usepackage{amssymb}
\usepackage{txfonts}
\usepackage{amsfonts}
\usepackage{amsmath}
\usepackage{graphicx}
\usepackage{hyperref}
\usepackage{float}
\usepackage{epstopdf}
\usepackage{color}

\allowdisplaybreaks
\newtheorem{theorem}{Theorem}[section]
\newtheorem{proposition}[theorem]{Proposition}
\newtheorem{lemma}[theorem]{Lemma}
\newtheorem{corollary}[theorem]{Corollary}

\def\FRAME#1#2#3#4#5#6#7#8
{
 \begin{figure}[ptbh]
 \begin{center}
 \includegraphics[height=#3]{#7}
 \caption{#5}
 \label{#6}
 \end{center}
 \end{figure}
}

\begin{document}

\title{On a recursive construction of Dirichlet form \\ on the Sierpi\'nski gasket}

\pagestyle{plain}

\author {Qingsong Gu}
\address{Department of Mathematics\\ The Chinese University of Hong Kong, Hong Kong, China}
\email{qsgu@math.cuhk.edu.hk}
\author {Ka-Sing Lau}
\address{Department of Mathematics\\ The Chinese University of Hong Kong, Hong Kong, China\\
\& Department of Mathematics, University of Pittsburgh, Pittsburgh, Pa. 15217, U.S.A.}
\email{kslau@math.cuhk.edu.hk}
{\author {Hua Qiu}}
\address {Department of Mathematics\\ Nanjing University, Nanjing, China}
\email {huaqiu@nju.edu.cn}

\thanks {The research of the first two authors were supported in part by the HKRGC grant; the second author was also support by NNSF of China (no. 11371382);  the third author was supported by NSF of China (no. 11471157).}

\maketitle
\begin{abstract}  {Let $\Gamma_n$ denote the $n$-th level Sierpi\'nski graph of the Sierpi\'nski gasket $K$. We consider, for any given conductance $(a_0, b_0, c_0)$ on $\Gamma_0$,  the Dirchlet form ${\mathcal E}$ on $K$ obtained from a recursive construction of compatible sequence of conductances $(a_n, b_n, c_n)$ on $\Gamma_n, n\geq 0$. We prove that there is a dichotomy situation: either $a_0= b_0 =c_0$ and  ${\mathcal E}$ is the standard Dirichlet form, or $a_0 >b_0 =c_0$ (or the two symmetric alternatives),  and  ${\mathcal E}$ is a non-self-similar Dirichlet form independent of $a_0, b_0$. The second situation has also been studied in \cite {HHW, HJ} as a one-dimensional asymptotic diffusion process on the Sierpi\'nski gasket. For the spectral property, we give a sharp estimate of the eigenvalue distribution of the associated Laplacian, which improves a similar result in \cite {HJ}. }
\end{abstract}

\maketitle
\section{\bf Introduction}
\setcounter{equation}{0}\setcounter{theorem}{0}

Dirichlet forms play a central role in the analysis on fractals. There is a large literature on the topic based on Kigami's analytic approach on the {\em post critically finite (p.c.f.) self-similar sets}, and the probabilistic approach of  Lindstr\o m on the nested fractals as well as Barlow and Bass on the Sierpi\'nski carpet (see \cite{B,BB, BBKT, FS, HMT, J, K1, K2, L, Pe1,PP, S} and the references therein).  In those studies, the  Sierpi\'nski gaskets and carpets are always served as fundamental examples, and are a source of inspiration.

\medskip

Recall that a {\it Sierpi\'nski gasket} (SG) is the unique nonempty compact set $K$ in $\mathbb{R}^2$ satisfying $K=\bigcup_{i=1}^3F_i(K)$ for an iterated function system (IFS)  $\{F_i\}_{i=1}^3$  on $\mathbb{R}^2$ such that $F_i(x)=\frac{1}{2}(x-p_i)+p_i$ with non-collinear $p_i$'s. For convenience, we fix $p_1=0$, $p_2=1$,  $p_3=\exp\left(\frac{\pi\sqrt{-1}}{3}\right)$.  Denote by  $V_0=\{p_1,p_2,p_3\}$ the {\em boundary} of $K$, and  let  $F_\omega=F_{\omega_1}\circ\cdots\circ F_{\omega_n}$ for a word $\omega\in W_n=\{1,2,3\}^n$.
The {\em standard Dirichlet form} ($\mathcal{E},\mathcal{F}$) on the SG is well-known  \cite{K, S}:  the energy ${\mathcal E}$ and the domain  ${\mathcal F}$ are given by
\begin{align} \label{eq1.1}
\mathcal{E}(u)=\lim\limits_{n\rightarrow\infty}\left(\frac53\right)^n\sum\limits_{p\sim_n q}(u(p)-u(q))^2,\quad
\mathcal{F}&=\{u\in C(K):\ \mathcal{E}(u)<\infty\}
\end{align}
where $p\sim_n  q$ means  $p \not=q$ and $p,q\in F_{\omega}(V_0)$ for some $\omega \in W_n$.  The domain $\mathcal{F}$ is known to be some Besov type space \cite{J}. In  \cite{SOB}, Sabot  classified all the Dirichlet forms on the SG which satisfy the energy self-similar identity
 \begin{equation} \label {eq1.2}
 \mathcal{E}(u)=\sum_{i=1}^3\frac1{r_i}\mathcal{E}(u\circ F_i),
 \end{equation}
where $r_i$, $i=1,2,3$ are some positive numbers called the {\em renormalization factors} of the energy form. The energy self-similar identity for the p.c.f fractals and nested fractals is also studied in detail in \cite{K}.

\medskip

More generally, one can also consider Dirichlet form without satisfying the energy self-similar identity. Let $\Gamma_0$ be the complete graph on $V_0$ and for $n\geq 1$, $\Gamma_n$  the graph on $V_n$ which is defined inductively by $V_n=\bigcup_{i=1}^3 F_i (V_{n-1})$ with the edge relation $\sim_n$ defined as in \eqref{eq1.1}.  Let $l(V_n)$ be the collection of functions defined on $V_n$, and let  $(\mathcal{E}_n, l(V_n))$ be defined by
\begin{equation} \label{eq1.3}
\mathcal{E}_n (u)={\sum}_{p\sim_n q}c^{(n)}_{pq}(u(p)-u(q))^2,   \quad u \in \ell (V_n),
\end{equation}
where  $c^{(n)}_{pq} \geq 0 $, call it the {\it conductance} of $p$ and $q$ in $\Gamma_n$. In the case that $\mathcal{E} (u) := \lim_{n\to \infty} \mathcal{E}_n (u) < \infty$ exists for  $u$ on $V_* = \bigcup_{n=0}^\infty V_n$, it will allow us to define a Dirichlet form on the SG.  For the limit to exist, the key issue is that the sequence of ${\mathcal E}_n$'s are compatible: {\it the restriction of ${\mathcal E}_n$ to $\ell (V_{n-1})$ must be equal to ${\mathcal E}_{n-1}, \ n \geq 1$}.

\medskip

In an attempt to produce all the Dirichlet forms (include the  non-self-similar ones),  Meyers, Strichartz and Teplyaev \cite {MST} used the compatibility condition to solve a system of linear equations of  conductances  on $V_1$ ($9$ of them) in terms of those on  $V_0$  as well as the given values of the harmonic functions on $V_1\setminus V_0$, then extend this inductively.  However the setup is too general and the  expressions are rather complicated, it does not give much information on the structure of the limiting Dirichlet form.  Recently two of the authors studied some anomalous p.c.f. fractals  in regard to the domains of the  Dirichlet forms and the associated Besov spaces \cite {GL}. In their investigation, a construction of the non-self-similar energy form was considered, and some interesting properties were found (see Section 4). In this note we intend to use the SG to study this construction  in greater detail  so as to give more insight to the general cases.

\bigskip

For this class of Dirichlet form on the SG,  we require the conductances of the  cells $F_\omega(V_0)$ on the same level $|\omega| = n$ are the same, and we will give a necessary and sufficient condition for the existence of a compatibility sequence  $\{{\mathcal E}_n\}_n$. The tool we use is the well-known electrical network theory.  The energy   $\mathcal{E}_n (u)$ in \eqref{eq1.3} corresponds to an electrical network $R(\Gamma_n)$ with {\it resistance} $r^{(n)}_{pq}= (c^{(n)}_{pq})^{-1}$, and $u$ is the {\it potential} on $V_n$. The sequence of networks $\{ R(\Gamma_n)\}_{n=0}^\infty$ are said to be  {\it compatible} if the resulting resistance of $R(\Gamma_n)$ on $V_{n-1}$ equals  $R(\Gamma_{n-1}), \ n\geq 1$. Note that this is equivalent to the compatibility of the sequence of energy forms ${\mathcal E}_n, n\geq 0$.

\bigskip

Let $(a_0, b_0, c_0)$ be the conductance on $V_0$, and let $(a_n, b_n, c_n)$ be the conductances of $F_\omega (V_0), |\omega|=n, n \geq 1$ to be determined. By the well-known $\Delta-Y$ transform \cite {K,S}, the resistances $(a_n^{-1}, b_n^{-1}, c_n^{-1})$ on the $\Delta$-side is equivalent  to a set of  resistances  $(x_n,y_n, z_n)$ on the $Y$-side.  It is direct to show (use \eqref{eq2.1} and refer to Figure \ref{fig2}) that $\{R(\Gamma_n)\}_{n=0}^\infty$  are compatible   can be reduced to  $\{(x_n,y_n,z_n)\}_{n\geq0}$ satisfy
\begin{equation}\label{eq1.4}
\begin{cases}
x_{n-1}=x_n + \phi(x_n; y_n, z_n),\\
y_{n-1}=y_n+\phi(y_n; z_n, x_n),\\
z_{n-1}=z_n+\phi(z_n; x_n, y_n),
\end{cases}  \qquad n\geq1 ,
\end{equation}
 where $\phi(x_n; y_n, z_n) := \frac{(x_n+y_n)(x_n+z_n)}{2(x_n+y_n+z_n)}$,  and symmetrically for the other two.
 We will refer to finding the solution of $(x_n, y_n, z_n)$ from $(x_{n-1}, y_{n-1}, z_{n-1})$  as a {\it recursive construction} of the energy form ${\mathcal E}_n$.  Necessarily, $(x_n,y_n,z_n)$ has to be positive, and the following is a necessary and sufficient condition for this to hold.  Let $\mathcal{E}^{(a_0, b_0)}_n := \mathcal{E}_n$ be defined as in \eqref{eq1.3} with conductances $(a_n, b_n, c_n)$ on each $n$-level subcells.

\bigskip

\begin {proposition} \label {th1.1}  For $a_0, b_0, c_0 >0$, in order for \eqref {eq1.4} to have positive solutions $(x_n, y_n, z_n), n\geq 1$,  it is necessary and sufficient that  $ x_0 \geq y_0=z_0> 0$ (or the symmetric alternates).

In this case, $  x_n \geq y_n =  z_n >0, \ n\geq 0$ and  $\{(x_n,y_n,z_n)\}_{n\geq0}$ is uniquely determined by the initial data $(x_0,y_0,z_0)$.
\end{proposition}

\medskip

 The proposition will be proved in Lemmas \ref{th2.1}, \ref {th2.2}.  We let  $\mu$ be the normalized $\alpha$-Hausdorff measure on  $K$ with $\alpha = \frac {\log 3}{\log 2}$. For two functions $f, g\geq 0$,  we use  $f \asymp g$ to mean that they dominate each other by a positive constant.    As a consequence of Proposition \ref{th1.1}, we have the following theorem.

\medskip

\begin{theorem} \label {th1.2} For the case $x_0>y_0=z_0>0$ in the above proposition,  we have $a_0 >b_0 =c_0$ and
\begin{equation*}
 a_n=\frac{x_n}{y_n(2x_n+y_n)}\asymp 2^n, \qquad
b_n = c_n= \frac 1 {2x_n+y_n}\asymp \left(\frac32\right)^n.
\end{equation*}
Moreover  ${\mathcal E}^{(a_0,b_0)}(u) =\lim\limits_{n\rightarrow\infty} {\mathcal E}^{(a_0,b_0)}_n(u)$ defines a strongly local regular Dirichlet form on $L^2(K, \mu)$ with domain ${\mathcal F}$ independent of $(a_0, b_0)$; it satisfies
\begin{equation} \label{eq1.5}
\mathcal{E}^{(a_0,b_0)}(u)=\sum\limits_{i=1}^3\mathcal{E}^{(a_1,b_1)}(u\circ F_i),
\end{equation}
but  does not satisfy the energy self-similar identity.
\end{theorem}

\bigskip

 It follows that for initial data $x_0\geq y_0=z_0>0$ on $\Gamma_0$, the recursive construction gives a dichotomy result on the Dirchlet forms:  when $a_0= b_0=c_0>0$,  then $\mathcal{E}^{(a_0,b_0)}$ is the standard Dirichlet form in \eqref{eq1.1};  when $a_0> b_0 (=c_0) >0$, then by the above estimation of $a_n$ and $b_n (= c_n)$, we have
\begin{align*}
&{\mathcal E}^{(a_0,b_0)}(u)\asymp\\
 &{\small \sup\limits_{n\geq0}\ \left\{ 2^n\sum\limits_{\omega\in W_n}\left (\Big(u_\omega(p_2)-u_\omega(p_3)\Big)^2 + \left(\frac34\right)^n\Big(u_\omega(p_1)-u_\omega(p_2)\Big)^2+\left(\frac34\right)^n\Big(u_\omega(p_1)-
 u_\omega(p_3)\Big)^2 \right )\right \} ,}
\end{align*}
where $u_\omega(x)=u\circ F_\omega(x)$.
It is seen that there  are two scaling factors in ${\mathcal E}^{(a_0, b_0)}$. The  renormalizing factor is $2^n$, and the energy is basically concentrated on the $\overline {p_2p_3}$ direction.

\bigskip

For this Dirichlet form ${\mathcal E}^{(a_0, b_0)}$ with $a_0>b_0$,  we can give a sharp estimate of the distribution of the eigenvalues (Section 3).  Let $\Delta^{(a_0,b_0)}$ be the {\em Laplacian}, the infinitesimal generator of $\left(\mathcal{E}^{(a_0,b_0)},\mathcal{F}\right)$ on $L^2(K,\mu)$.
Denote by $\rho^{(a_0,b_0)}(t)$ the eigenvalue count with the  {\it Dirichlet boundary condition} (D.B.C), that is
\begin{equation} \label{eq1.6}
\rho^{(a_0,b_0)}(t)=\#\Big\{\lambda \leq t: \lambda \text{ is an eigenvalue of $-\Delta^{(a_0,b_0)}$ with D.B.C.}\Big\},
\end{equation}

\medskip

\begin {theorem} \label {th1.3}
Assume that $a_0>b_0=c_0$,  and let  $t_0 = \inf\{t: \rho^{(a_0,b_0)}(t) >0\}$, then
\begin{equation*}
 \rho^{(a_0,b_0)}(t) \asymp t^{\frac{\log3}{\log(9/2)}}, \qquad  t > t_0.
\end{equation*}
\end{theorem}

\bigskip

We remark that in another investigation, K. Hattori, T. Hattori and Watanabe \cite {HHW} studied the asymptotically  one-dimensional diffusion processes on the SG (see also Hambly and Jones \cite{HJ}, Hambly and Yang \cite{HY}). The random walk they considered is in fact  the normalized probability of $(a_n, b_n, b_n)$  as transition probability on the three sides of the $n$-level cells of the SG. (They used this as an assumption, and in fact it is one of the dichotomy cases from Theorem \ref {th1.2} (or Proposition \ref{th1.1}).)  We will give a brief comparison of these two approaches in Section 2.   For the estimate of the eigenvalue distribution in  Theorem \ref{th1.3}, it improves  the lower bound  of $\rho^{(a_0, b_0)}(t)$  in  \cite  [Theorem 13] {HJ}  where it was shown to be $C^{-1}t^{\log 3 / \log (9/2)}(\log t)^{-\beta} $ with $\beta > \log 3/\log 2$, using a heat kernel technique in the estimation.

\bigskip

The recursive construction can be extended to more general p.c.f. sets (see \cite {GL} for some examples), but it also have limitation. In Section 4, we give two other examples that this construction have abnormality. The first one is the {\it twisted SG} introduced by Mihai and Strichartz \cite {MS}, it is a modification of the IFS of the SG that reflecting  the three subcells of the SG along the angle bisectors at the three vertices. We show that  for $a_0 > b_0 = c_0$, the closure of $V_*$ under the (effective) resistance metric has interesting topology different from the SG; the second one is from \cite {GL}, it is called a {\it Sierpinski sickle},  which  is the attractor of an  IFS of $17$ similitudes and three boundary points, of which the recursive construction does not yield a compatible sequence for a Dirichlet form.

\bigskip
\bigskip

\section{\bf Proof of Theorem \ref{th1.2}}
\setcounter{equation}{0}\setcounter{theorem}{0}
\bigskip

Let $(a, b, c)$ denote the conductance of a $\Delta$-shape network. Recall the {\em $\Delta$-Y transform} (see e.g., \cite {K}, \cite{S}) states that
the $\Delta$-shaped network with resistance $(a^{-1}, b^{-1}, c^{-1})$ and the Y-shaped network with resistance $(x,y,z)$  (see Figure \ref{fig1})  are equivalent by the following relation
\begin{equation} \label{eq2.1}
  x =\frac{a}{\eta} , \quad  y =\frac{b}{\eta},\quad  z  =\frac c{\eta}\ ,
\end{equation}
with $\eta = ab+bc +ca$, and conversely,
\begin{equation} \label{eq2.2}
a = \frac xr, \quad b = \frac yr, \quad c = \frac zr \ ,
\end{equation}
where $r = xy+ yz+ zx = \eta^{-1}$.

\begin{figure}[h]
\textrm{\centering
\scalebox{0.18}[0.18]{\includegraphics{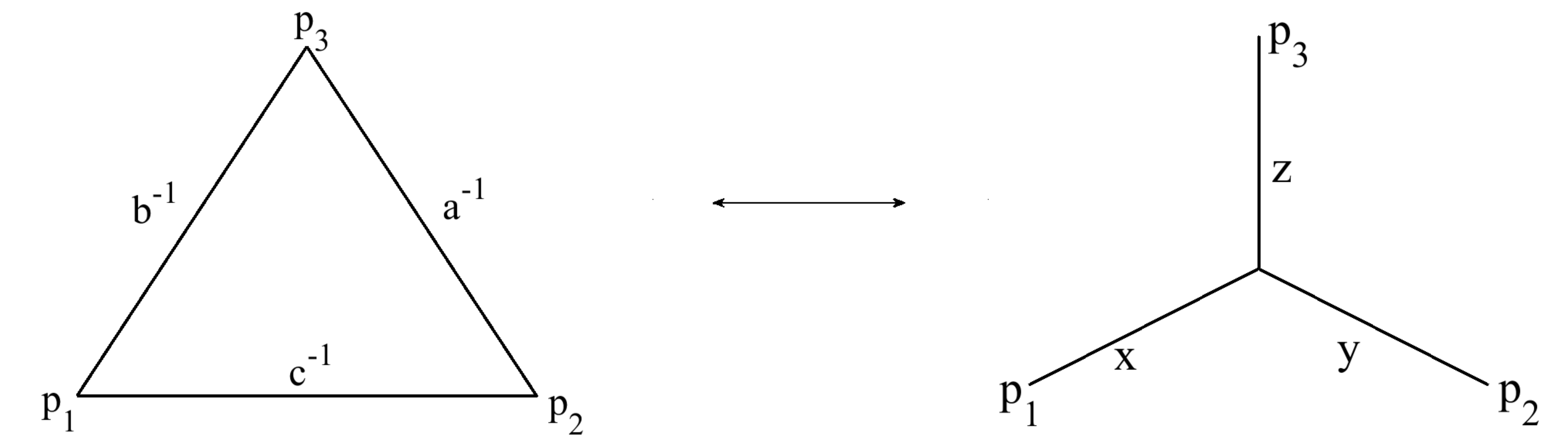}}
}
\caption{$\Delta-Y$-transform}
\label{fig1}
\end{figure}

\bigskip

Assume the conductances on the edges of the $n$-th level cells are given by $(a_n,b_n,c_n)$ for $n\geq0$.  The compatibility of the $n$-th and $(n-1)$-th resistance networks on the $Y$-side reduces to calculation the resulting resistance on the left side in Figure \ref{fig2}, which yields \eqref {eq1.4}.

\begin{figure}[h]
\textrm{\centering
\scalebox{0.18}[0.18]{\includegraphics{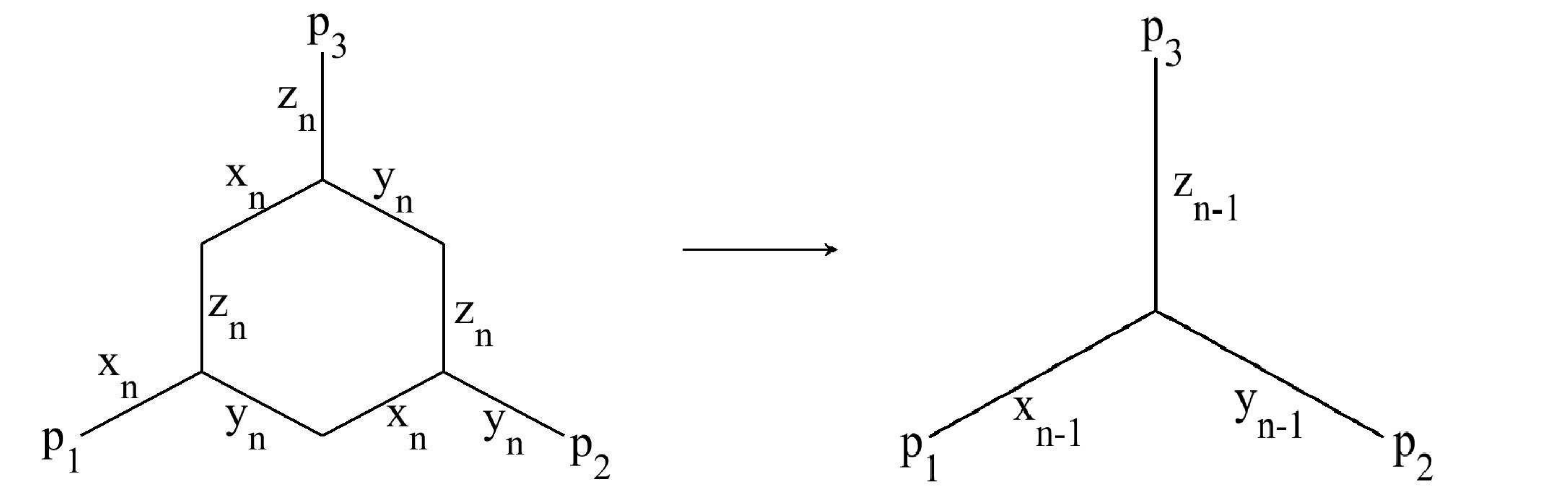}}
}
\caption{Consistence of the $n$-th and $(n-1)$-th resistance networks }\
\label{fig2}
\end{figure}

\bigskip

Our first lemma is to characterize all compatible resistance sequences $\{(x_n,y_n,z_n)\}_{n\geq 0}$.

\bigskip

\begin{lemma}\label{th2.1}
In order for
(\ref{eq1.4})  to have positive solutions $(x_n, y_n, z_n)$  for all $n\geq1$, it is necessary and sufficient that $x_0\geq y_0 = z_0>0$ (or its symmetric alternatives).
\end{lemma}

\bigskip

\begin{proof}
{\it Sufficiency}. Without loss of generality, assume that $x_0\geq y_0=z_0>0$. Then using this to solve the equations \eqref{eq1.4}, we have
\begin{equation}\label{eq2.3}
\begin{cases}
x_1=\frac1{15}\left(14x_0+3y_0-2\sqrt{4x_0^2+6x_0y_0+6y_0^2}\right),\\
y_1(=z_1)=\frac15\left(-2x_0+y_0+\sqrt{4x_0^2+6x_0y_0+6y_0^2}\right),
\end{cases}
\end{equation}
is a pair of positive solution of (\ref{eq1.4}). Also by  $x_0\geq y_0$, we have
\begin{equation*}
x_1-y_1=\frac1{15}\left(20x_0-5\sqrt{4x_0^2+6x_0y_0+6y_0^2}\right)\geq0.
\end{equation*}
 Hence, $x_1 \geq y_1 =z_1 $. We can repeat this process inductively, and obtain the sequence $\{(x_n,y_n,z_n)\}_{n\geq 0}$ as positive solution of \eqref{eq1.4}.

\medskip

{\it Necessity}. Without loss of generality, let $x_0\geq y_0\geq z_0>0$, we will show that $y_0 = z_0$.  Assume otherwise, $y_0 > z_0$. Let $(x_1,y_1,z_1)$ be positive solution
of (\ref{eq1.4}) for $n=1$,   we first prove the following  claims in regard to $(x_1,y_1,z_1)$:

\medskip

 \ (i) \ $x_1\geq y_1>z_1$:  \ \ For if  $x_1<y_1$, then clearly, $x_1+\phi (x_1; y_1, z_1) < y_1+\phi (y_1; z_1, x_1)$, which is $x_0<y_0$, a contradiction. Hence $x_1\geq y_1$;  by the same argument,  we have  $y_1>z_1$ from $y_0>z_0$.

\medskip

 (ii)\ \ $\dfrac{y_1}{z_1}>\dfrac{y_0}{z_0}$: \ \
Indeed, if this were not true, letting $\frac{y_0}{z_0}=\mu_0>1$, we have
\begin{equation*}
\frac{y_1}{z_1}\leq\mu_0=\frac{y_0}{z_0}=\frac{y_1+\phi(y_1; z_1, x_1)}{z_1+
\phi (z_1; x_1, y_1)}\ .
\end{equation*}
 Therefore $\frac{y_1}{z_1}\leq\frac{y_1+x_1}{z_1+x_1}$, that is $y_1\leq z_1$, which contradicts the fact that $y_1> z_1$ in (i).

\medskip

(iii)\  Let  $\lambda_0=\frac{2x_0}{y_0+z_0}>1$, and let $\rho=\frac15\left(6-\lambda_0^{-1}\right) (>1)$, we claim that
\begin{equation} \label{eq2.4}
\frac{2x_1}{y_1+z_1}\geq\lambda_0\rho.
\end{equation}
If otherwise, then
\begin{equation}\label{eq2.5}
\frac{2x_1}{y_1+z_1}<\lambda_0\rho.
\end{equation}
By (\ref{eq1.4}), we have
\begin{equation}\label{eq2.6}
\lambda_0=\frac{2x_0}{y_0+z_0}=\frac{2x_1+2 \phi(x_1; y_1, z_1)}
{(y_1+z_1)+\frac{(y_1+z_1)(2x_1+y_1+z_1)}{2(x_1+y_1+z_1)}}.
\end{equation}
Observe that
$$
\frac{2(x_1+y_1+z_1)}{(2x_1+y_1+z_1)}<2.
$$
This, together with (\ref{eq2.5}), (\ref{eq2.6}) and a simple calculation, yields
\begin{equation}\label{eq2.7}
\frac{2(x_1+y_1)(x_1+z_1)}{(y_1+z_1)(2x_1+y_1+z_1)} > \lambda_0\cdot(3-2\rho).
\end{equation}
On the other hand, by using $\rho=\frac15\left(6-\lambda_0^{-1}\right)$, we have
\begin{align*}
\frac{2(x_1+y_1)(x_1+z_1)}{(y_1+z_1)(2x_1+y_1+z_1)}
&\leq\frac12\cdot\frac{(2x_1+y_1+z_1)^2}{(y_1+z_1)(2x_1+y_1+z_1)}=\frac12\left(1+\frac{2x_1}{y_1+z_1}\right)\\
&<\frac12(1+\lambda_0\cdot\rho)=\lambda_0\cdot(3-2\rho).
\end{align*}
 This contradicts (\ref{eq2.7}), and (iii) follows.

 \medskip

By (i), we can carry out the estimate in (iii) inductively and obtain
\begin{equation} \label{eq2.8}
\frac{2x_n}{y_n+z_n}\geq\lambda_0\cdot\rho^{n}\rightarrow\infty, \quad \hbox { as}  \ n \to \infty.
\end{equation}
Also using (ii),  we have $\frac{y_n}{z_n}\geq\frac{y_0}{z_0}=\mu_0>1$ for any $n\geq1$,  and a similar argument as in (iii) yields
\begin{equation}\label{eq2.9}
\frac{y_n}{z_n}\rightarrow\infty, \quad \hbox {as} \  n\rightarrow\infty
\end{equation}
(for example, one can take $\rho=(5-\mu_0^{-1})/4 > 1$, and show that there is $n_0$ such that for all $n\geq n_0$,
$
\frac{y_n}{z_n}> \rho\frac{y_{n-1}}{z_{n-1}}
$
holds).

\medskip
Now consider
\begin{equation*}
\begin{cases}
y_{n-1}=y_n+\phi (y_n; z_n, x_n),\\
z_{n-1}=z_n+ \phi (z_n; x_n, y_n),
\end{cases}
\end{equation*}
for $n$ and $\frac{x_n}{y_n+z_n}$  sufficiently large.
By $\frac{x_n}{y_n+z_n}\rightarrow\infty$, it reduces to
\begin{equation*}
\begin{cases}
y_{n-1}=\big(\frac 32 y_n+ \frac 12 z_n\big) (1+ o(1)) ,\\
z_{n-1}=\big (\frac 32 z_n+ \frac 12 y_n\big ) (1+o(1)),
\end{cases}
\end{equation*}
where $o(1)$ is an error term that tends to $0$ as $n\to \infty$. Therefore we obtain
\begin{equation*}
\begin{cases}
y_{n}\asymp \frac34y_{n-1}-\frac14z_{n-1},\\
z_{n}\asymp \frac34z_{n-1}-\frac14y_{n-1}.
\end{cases}
\end{equation*}
This together  with (\ref{eq2.9}) contradicts the assumption that $\{z_n\}_{n\geq0}$ are positive. Therefore we must have $y_0 = z_0$, and completes the proof.
\end{proof}

\bigskip

\begin{lemma}\label{th2.2}
Let $x_0\geq y_0 = z_0 >0$ be fixed, then for $n\geq1$,
\begin{equation*}
\begin{cases}
x_n=\frac1{15}\left(14x_{n-1}+3y_{n-1}-2\sqrt{4x_{n-1}^2+6x_{n-1}y_{n-1}+6y_{n-1}^2}\right),\\
y_n=\frac15\left(-2x_{n-1}+y_{n-1}+\sqrt{4x_{n-1}^2+6x_{n-1}y_{n-1}+6y_{n-1}^2}\right).
\end{cases}
\end{equation*}
Also for  $x_0 = y_0=z_0$, then $x_n= y_n=z_n= \left (\dfrac 35\right)^n x_0$,  and for  $x_0 > y_0=z_0$,
$$
x_n\asymp\left(\dfrac23\right)^n, \quad y_n= z_n \asymp\left(\dfrac12\right)^n.
$$
\end{lemma}

\medskip

\begin{proof}  Similar to \eqref {eq2.3}, we can solve equations \eqref{eq1.4} for $x_n$ and $y_n$ as the above. It follow that if $x_0=y_0=z_0$, then $x_n=y_n=z_n = \left (\dfrac 35\right)^n x_0$.  By Lemma \ref{th2.1}, we see that $x_0>y_0  = z_0 >0$ implies $x_n> y_n = z_n$  inductively.
Also from \eqref{eq2.8}, we see that for all $n\geq0$, $\frac{y_n}{x_n}\leq C\delta^n$ for some constant $C>0$ and $0<\delta<1$ (depending only on $\frac{y_0}{x_0}$). Combining this with
\begin{equation*}
\frac{x_n}{x_{n-1}}=\frac1{15}\left(14+3\frac{y_{n-1}}{x_{n-1}}-2\cdot\sqrt{4+
6\frac{y_{n-1}}{x_{n-1}}+6\left(\frac{y_{n-1}}{x_{n-1}}\right)^2}\right),
\end{equation*}
we can find  $C_1>0$ such that for $n$ large,
\begin{equation*}
\frac23-C_1\delta^n\leq\frac{x_n}{x_{n-1}}\leq\frac23+C_1\delta^n.
\end{equation*}
Therefore we have
$
x_n\asymp \left(\frac23\right)^n,
$ and
similarly $
y_n\asymp \left(\frac12\right)^n.
$
\end{proof}

\bigskip

\noindent {\bf Proof of Proposition \ref{th1.1}}. It  follows readily from Lemmas \ref{th2.1} and \ref{th2.2}. \hfill $\Box$

\bigskip

It follows from the compatibility of $\{(x_n, y_n, z_n)\}_{n\geq 0}$ and the $\Delta$-$Y$ transform  that $\{{\mathcal E}_n^{(a_0, b_0)}\}_{n\geq 0}$ are compatible. Hence for  a function $u \in \ell (V_n)$, we can construct inductively  harmonic extensions $ u_m$ on $V_m, \ m>n$ and ${\mathcal E}^{(a_0, b_0)}_m( u_m) = {\mathcal E}^{(a_0, b_0)}_n( u)$; also for $u\in \ell (V_*)$, ${\mathcal E}_n(u|_{V_n})$ is an increasing sequence.  We  define ${\mathcal E} (u) := {\mathcal E}^{(a_0, b_0)} (u) = \lim_{n \to \infty} {\mathcal E}_n^{(a_0, b_0)} (u|_{V_n}) $ for $u \in \ell (V_*)$.  Recall that the {\it (effective) resistance metric} $R := R^{(a_0,b_0)}$ on $V_*\times V_*$ is defined by   $R^{(a_0, b_0)}(x,x)=0$ for any $x\in V_*$ , and for any two distinct points $x,y\in V_*$,
\begin{equation*}
R(x,y)^{-1}:=\inf\{\mathcal{E}(u):\ u\in\ell (V_*), \  u(x)=1, u(y)=0\}.
\end{equation*}
\bigskip
Note that  for $a_0 =b_0=c_0$,  then $R(x, y)  \asymp |x-y|^{\gamma}$ where  $\gamma=\frac{\log(5/3)}{\log2}$ \cite {K, S}.
\bigskip

\medskip

\begin{proposition} \label{th2.3}  For  $a_0 > b_0 =c_0$, the completion of the $(V_*, R^{(a_0, b_0)})$   is $K$,  and
\begin{equation}\label{eq2.10}
C^{-1}|x-y|\leq R^{(a_0,b_0)}(x,y)\leq C|x-y|^{\gamma'},  \qquad  x, y \in K
\end{equation}
where $\gamma'=\frac{\log3}{\log2}-1$ and $C>0$ is a constant  depends on $a_0$ and $b_0$.

Furthermore $R^{(a_0, b_0)} $ is a bounded metric with
\begin{equation}\label{eq2.11}
\sup\left\{R^{(a_0,b_0)}(x,y): x,y\in K\right\}\leq C'b_0^{-1}.
\end{equation}
where $C' >0$ is independent of $a_0$ and $b_0$.
\end{proposition}

\medskip

\begin{proof}  Fix $x_0>y_0=z_0>0$,   then $x_n> y_n =z_n$.  As in \eqref{eq2.2},  $r_n = x_ny_n + y_nz_n + z_nx_n =2x_n y_n + y_n^2$.  By \eqref{eq2.2} and Lemma \ref{th2.2},
$$
a_n =  \dfrac {x_n}{r_n} = \dfrac {x_n} {2x_n y_n + y_n^2} \asymp 2^n,\qquad
b_n = c_n =  \dfrac {y_n}{r_n}=  \dfrac {y_n} {2x_n y_n + y_n^2} \asymp \Big (\dfrac 32\Big)^n.
$$

Let us write $R(x, y) = R^{(a_0, b_0)}(x, y)$.
To estimate $R(x, y)$  on $V_*$,  we first consider  $x\sim_n y$,  and  let  $\psi^{(n)}_x (z) = \delta_x (z) , \  x, z \in V_n$ where $\delta_{x} $ is the Dirac measure on $V_n$. It follows that
$$
R^{-1}(x, y)  \leq {\mathcal E} (\psi^{(n)}_x) \leq  C_1 2^n = C_1|x-y|^{-1}.
$$
On the the other hand, we have
$$
R^{-1}(x, y) \geq \min\{a_n, b_n\} \geq C_2 \Big (\frac 32 \Big )^n = C_2  |x-y|^{-\log (3/2)/ \log 2}.
$$
 For the estimate of $R(x,y)$ with any distinct $x,y\in V_*$. Let $n$ be the maximal integer such that both $x,y$ belong to either an $n$-level cell or a union of two adjacent $n$-level cells. Then using a similar argument as above, we have $R(x,y)\leq C\Big (\dfrac 23\Big)^n$ and $R(x,y)\geq C^{-1}2^{-n}$. This gives that $R(x,y)$ satisfies the required estimate since $|x-y|\asymp 2^{-n}$.  This completes the proof of \eqref {eq2.10} for $x, y \in V_*$,  it  follows that the completion of $(V_*, R)$ is $K$, and the same estimate holds for $x,y \in K$.

\medskip

To prove \eqref{eq2.11}, we only need to estimate $R(x,p_1)$ from above with $x\in K$ since for any two points $x,y$ in $K$, $R(x,y)\leq R(x,p_1)+R(y,p_1)$. We can find a chain of points $\{x_n\}_{n=0}^\infty$ in $V_*$ with $x_0=p_1$ and $x_n\rightarrow x$ as $n\rightarrow\infty$ such that $x_n,x_{n+1}$ are two of the boundary points of some $(n+1)$-cell.
Thus by triangle inequality, we have
\begin{equation}\label{eq2.12}
R(x,p_1)\leq \sum\limits_{n=0}^{\infty}R(x_n,x_{n+1})\leq \sum\limits_{n=0}^{\infty}b_n^{-1}.
\end{equation}
On the other hand, we see that
{\small \begin{align*}
&\frac{b_{n-1}}{b_n}=\frac{2x_n+y_n}{2x_{n-1}+y_{n-1}}\leq\max\left\{\frac{x_n}{x_{n-1}},
\frac{y_n}{y_{n-1}}\right\}\\
=&\max\left\{\frac1{15}\left(14+3\frac{y_{n-1}}{x_{n-1}}-2\sqrt{4+
6\frac{y_{n-1}}{x_{n-1}}+6\Big(\frac{y_{n-1}}{x_{n-1}}\Big)^2}\right),\ \frac1{5}
\left(-2\frac{x_{n-1}}{y_{n-1}}+1+
\sqrt{4\Big(\frac{x_{n-1}}{y_{n-1}}\Big)^2+
6\frac{x_{n-1}}{y_{n-1}}+6}\right)\right\}\\
\leq&\max\left\{\frac1{15}(14+3-2\sqrt{4}),\ \frac15\left(-2\frac{x_{n-1}}{y_{n-1}}
+1+2\frac{x_{n-1}}{y_{n-1}}+3\right)\right\}=\frac{13}{15}<1.
\end{align*}}
Therefore the series in (\ref{eq2.12}) converges and is bounded above by $Cb_0^{-1}$.

\vspace {0.1cm}
\end{proof}

\bigskip

  It follows that under the resistance metric, $u\in \ell (V_n)$ can be extended harmonically on $V_*$, then continuously on $K$,  we call this an {\it $n$-piecewise harmonic function} on $K$. As a special case, consider the  harmonic function  that takes value $1, 0,0$ on $p_1, p_2, p_3$. It is direct to check,  using the harmonicity of $u$ at $V_1 \setminus V_0$, that
\begin{equation} \label{eq2.13}
u(p_{12}) =u(p_{13}) = \frac {a_1 +b_1}{3a_1 + 2b_1}, \qquad u(p_{23})= \frac {b_1}{3a_1 + 2b_1}.
\end{equation}
For the special case that $a_0=b_0 =c_0$, it is the $\frac 15$-$\frac25$-law in the standard Dirichlet form on SG \cite {K, S}.

\bigskip

\noindent {\bf Proof of Theorem \ref {th1.2}.}  Fix $x_0>y_0=z_0>0$,   it follows from the proof in Proposition \ref {th2.3} that
$
a_n \asymp 2^n,\
b_n = c_n  \asymp \Big (\dfrac 32\Big)^n.
$
For $u\in C(K)$ and $n\geq0$, let
\begin{equation*}
{\mathcal E}^{(a_0,b_0)}_n(u) = \sum\limits_{\omega\in W_n}b_n\Big(u_\omega(p_1)-u_\omega(p_2)\Big)^2+b_n\Big(u_\omega(p_1)-u_\omega(p_3)\Big)^2+a_n\Big(u_\omega(p_2)-u_\omega(p_3)\Big)^2,
\end{equation*}
where $u_\omega(x)=u\circ F_\omega(x)$.
Define
\begin{align*}
{\mathcal E}^{(a_0,b_0)}(u) =\lim\limits_{n\rightarrow\infty} {\mathcal E}^{(a_0,b_0)}_n(u|_{V_n}), \quad
{\mathcal F}(={\mathcal F}^{(a_0,b_0)})=\{u\in C(K): {\mathcal E}^{(a_0,b_0)}(u)<\infty\}.
\end{align*}
In view of the compatibility of the sequence $\{(a_n, b_n, c_n)\}_n$, we have
$
{\mathcal E}^{(a_0,b_0)}_n(u) = \sum_{i=1}^3{\mathcal E}^{(a_1,b_1)}_{n-1}(u\circ F_i).
$
By taking limit, we obtain
\begin{equation*}
{\mathcal E}^{(a_0,b_0)}(u) = \sum_{i=1}^3{\mathcal E}^{(a_1,b_1)}(u\circ F_i).
\end{equation*}

 It is standard to check that $( {\mathcal E}^{(a_0,b_0)}, \mathcal F)$ is a Dirichlet form on $L^2(K, \mu)$.  It is regular by observing that  the {\em piecewise harmonic functions} are continuous functions in ${\mathcal F}$ and are dense in $C(K)$, and $C(K)\cap {\mathcal F} (= {\mathcal F})$ is trivially $({\mathcal E}^{(a_0,b_0)})^{1/2}+ ||\cdot ||_{L^2(K, \mu)}$-dense in ${\mathcal F}$.  By using the above identity  repeatedly, we obtain that for any $n\geq1$,
\begin{equation}\label{eq2.14}
{\mathcal E}^{(a_0,b_0)}(u) = \sum_{\omega\in W_n}{\mathcal E}^{(a_n,b_n)}(u\circ F_\omega),
\end{equation}
which leads to the strong locality of  $({\mathcal E}^{(a_0,b_0)},{\mathcal F})$.

\vspace {0.1cm}

Finally, we see that $({\mathcal E}^{(a_0,b_0)},{\mathcal F})$ does not satisfy the energy self-similar identity (\ref{eq1.2}). It is because  if it satisfies the identity for some $r_i$, then by our construction, all the $r_i$ in (\ref{eq1.1}) should be equal. However, by the uniqueness result of Sabot \cite{SOB}, $({\mathcal E}^{(a_0,b_0)},{\mathcal F})$ should be the standard one defined by (\ref{eq1.1}), a contradiction.
\hfill $\Box$

\bigskip
The following  dichotomic result  follows directly from Theorem \ref{th1.2}.

\medskip

\begin {corollary}  \label {th2.4}For the recursive construction of the Dirichlet form with initial data $(a_0, b_0, c_0)$, there are only two cases, either

(i) \ $a_0= b_0 =c_0$, and in this case  ${\mathcal E}$ is the standard Dirichlet form as in \eqref{eq1.1},

\noindent or

(ii) \ $a_0 > b_0 =c_0$ (or the symmetric alternates), and the Dirichlet form satisfies
\begin{align*}
&{\mathcal E}(u)\asymp\\
 &{\small \sup\limits_{n\geq0}\ \left\{ 2^n\sum\limits_{\omega\in W_n}\left (\Big(u_\omega(p_2)-u_\omega(p_3)\Big)^2 + \left(\frac34\right)^n\Big(u_\omega(p_1)-u_\omega(p_2)\Big)^2+\left(\frac34\right)^n\Big(u_\omega(p_1)-
 u_\omega(p_3)\Big)^2 \right )\right \} .}
\end{align*}
\end{corollary}

\bigskip

 It is well-known that a regular strongly local Dirichlet form  associates with a continuous diffusion process \cite {FOT}.  In fact, this probability counter part of ${\mathcal E}^{(a_0, b_0)}$ had been studied by Hattori {\it et al} \cite {HHW} as an asymptotically one-dimensional diffusion processes on the SG.  To conclude this section, we give a brief discussion of their study in comparison with our consideration.

 \medskip

 For a random walk  $\{Z^{(n, \alpha)}_k\}_k$  on $V_n$ with $\alpha = (\alpha_1, \alpha_2, \alpha_3)$, the probability that the walk goes to the four neighbors (except at $V_0$) in the three directions (counting the opposite direction as one), define the {\it $(n-1)$-decimated walk} $\{Z'_\ell\}_\ell $ on $V_{n-1}$ that records the visit of $Z^{(n, \alpha)}_k $ in $V_{n-1}$ in the $\ell$-th time (with a state distinct from $Z'_{\ell-1}$). Then it is direct to show that for $\{Z^{(n, \alpha)}_k\}_k $ with starting point on $V_{n-1}$, \  $\{Z'_\ell\}_\ell$ obeys the same law as  $\{Z^{(n-1, T\alpha)}_k\}_k$
where
 $$
 T\alpha = C\Big ( \alpha_1 + \frac {\alpha_2 \alpha_3}3, \ \alpha_2 + \frac {\alpha_3 \alpha_1}3, \ \alpha_3 + \frac {\alpha_1 \alpha_2}3 \Big),
$$
 and $C$ is a normalized constant  \cite {HHW}. This sets up the compatible condition by letting  $\alpha_{n-1} =T\alpha_n$ (renormalization group), the exact analog of \eqref{eq1.4}.  Then they define the random walk using
$$
\alpha_n = ( \alpha_{n,1},\  \alpha_{n,2},\  \alpha_{n,3} ) := C'( 1, w_n, w_n),
$$
where $0<w_0<1$, $C'$ is a normalized constant, and $w_n,\  n\geq 1 $,  are defined inductively by
$$
w_n =  \Big (-2 + 3w_{n-1} + \sqrt {4+6w_{n-1} + 6 w_{n-1}^2}\ \Big )\ \big / \ \big (6-w_{n-1}\big).
$$
For $x_n, y_n $ in Lemma \ref{th2.2},  it can be shown that $y_n/x_n$ has the same expression as the above $w_n$.

\medskip

Note that in this case  $\lim_{n\to \infty} \alpha_n = (1, 0, 0)$. Let $X_t (n) = Z_{[6^nt]}(n, \alpha_n)$, then with some more work, they proved that $ \{X_t (n)\}_{n=0}^\infty$ converges weakly to a continuous, strongly Markov processes $X_t$ on $K$, and the moves are asymptotically one-dimensional, dominated in the direction parallel to $\overline {p_2p_3}$,  and  of order $O(3/4)^n$ in the other two directions.  This is in line with the expression of ${\mathcal E}^{(a_0, b_0)}$ in Corollary \ref{th2.4}(ii), as the energy  has two scaling exponents and is concentrated in the $\overline {p_2p_3}$ direction.

\bigskip
\bigskip

\section{\bf Spectral asymptotics}
\setcounter{equation}{0}\setcounter{theorem}{0}

\bigskip

Let $\Delta^{(a_0,b_0)}$ be the {\em Laplacian}, the infinitesimal generator of the Dirichlet form $\left(\mathcal{E}^{(a_0,b_0)},\mathcal{F}\right)$ on $L^2(K,\mu)$. In both cases  $a_0=b_0$ and $a_0>b_0$, $\mathcal{F}$ is compactly imbedded in $C(K)$ and hence in $L^2(K,\mu)$. Therefore the eigenvalues of $-\Delta^{(a_0,b_0)}$ with the Dirichlet or Neumann boundary condition are nonnegative, countable and  have no limit point.
Denote by $\rho^{(a_0,b_0)}(t)$ the {\em eigenvalue counting function} of $-\Delta^{(a_0,b_0)}$ with the {\em Dirichlet boundary condition}  as in \eqref {eq1.6},
and by $\rho^{(a_0,b_0)}_N(t)$ the eigenvalue counting function of $-\Delta^{(a_0,b_0)}$ with the {\em Neumann boundary condition}, where in both cases, each eigenvalue is counted according to its multiplicity. We are interested in the asymptotic growth rate of $\rho^{(a_0,b_0)}(t)$ and $\rho^{(a_0,b_0)}_N(t)$ as $t\rightarrow\infty$.  It is known that  (see \cite[Lemma 2.3(2)]{KL93})
\begin{equation} \label{eq3.1}
\rho^{(a_0,b_0)}(t)\leq\rho^{(a_0, b_0)}_N(t)\leq \rho^{(a_0,b_0)}(t)+3,
\end{equation}
where $3$ is the dimension of the space of all the harmonic functions on $K$. Hence $\rho^{(a_0,b_0)}(t)$ and $\rho^{(a_0,b_0)}_N(t)$ have the same asymptotic behavior.

\bigskip

In the case  $a_0=b_0 =c_0$ for the standard Dirichlet form, it is known that (e.g. \cite{FS}, \cite{KL93})
\begin{equation*}
\rho^{(a_0,b_0)}(t) \asymp t^{\log3 /\log 5},  \qquad t\to \infty.
\end{equation*}

\medskip
In the following, our concentration is on the case  $a_0>b_0 =c_0$.   First we provide a general result on the dimension of some linear subspaces.  Recall that a linear subspace $\mathcal{L}$ of  $L^2(K,\mu)$ is called a {\it sublattice} if $u\in\mathcal{L}$ implies $|u|\in\mathcal{L}$.

\medskip

\begin{proposition}\label{th3.1}
Let $K$ be a  compact connected set and $\mu$ be a Borel measure on $K$ with full support, and let $(\mathcal{E},\mathcal{F})$ be a regular Dirichlet form on $L^2(K,\mu)$ with $\mathcal{F}\subset C(K)$. Denote by $\{P_t\}_{t\geq0}$ the associated semigroup of operators of $(\mathcal{E},\mathcal{F})$. Suppose $\mathcal{L}\subset {\mathcal F}$ is a closed linear sublattice of $L^2(K,\mu)$, and there exists  $C>0$ such that
\begin{equation}\label{eq3.2}
P_tu\leq Cu, \qquad \forall \ t>0, \ \  u \geq 0, \ \ u\in \mathcal{L}.
\end{equation}
Then $\mathcal{L}$ has dimension at most one.
\end{proposition}

\medskip

\begin{proof}
The essentially idea of the proof comes  from \cite[Theorems 7.2, 7.3]{D80}.
 Suppose $\mathcal{L}$ is nontrivial, let $u\geq0$ be any non-zero element in $\mathcal{L}$, then $u\in C(K)$.  Let $ U=\{x\in K:\ u(x)\neq0\}$. We claim that  $U=K$, modulo a $\mu$-null set. If $ v \in C(K)$ and $|v|\leq \alpha u$ for some $\alpha\geq0$, then by the Markovian property of $\{P_t\}_{t>0}$ and (\ref{eq3.2}), we have
 \begin{equation*}
 |P_tv|\leq P_t|v|\leq\alpha P_t u\leq \alpha C u.
 \end{equation*}
 Hence for
 \begin{equation*}
 {\mathcal G}=\{v\in C(K):|v|\leq\alpha u \  \text{ for some }\alpha\geq0\},
 \end{equation*}
then $P_t({\mathcal G})\subseteq {\mathcal G}$ for all $t\geq0$.
As $U$ is an open set by definition, ${\mathcal G}$ contains all the continuous functions that are compactly supported in $U$.  The $L^2$-closure of ${\mathcal G}$ is the set of all $v\in L^2(K,\mu)$ with $v =0$ on $K \setminus U$. So $U$ is an \emph{invariant} set of the semigroup $\{P_t\}_{t>0}$. (A $\mu$-measurable set $B\subset K$ is said to be $P_t$--invariant if $P_t(1_Bf)=1_BP_tf$
$\mu$-a.e. for any $f\in L^2$ and $t>0$.) Hence by \cite[Theorem 1.6.1]{FOT}, $1_U\in\mathcal{F}$. However, as $K$ is connected,  this holds if and only if $U=K$ or $U=\emptyset$. Since $u$ is nonzero, we conclude that $U=K$, and  the claim follows.

\vspace {0.1cm}

  Now, if $u\in\mathcal{L}$, then $u^+$ and $u^-$ are in $\mathcal{L}$ and have disjoint supports. It follows from the claim  that one of them must vanish. Hence $u\in\mathcal{L}$ implies  $u\geq0$ or $(-u)\geq0$. If $u,v$ are two distinct positive elements of $\mathcal{L}$, then $u+\eta v$ is either positive or negative for all $\eta\in\mathbb{R}$.  But the sum must change sign as $\eta$ increases through ${\Bbb R}$.  Hence there is $\eta$ such that $u+\eta v=0$. This is a contradiction,  and hence $\mathcal{L}$ is one dimensional.
\end{proof}

\medskip

\begin{lemma}\label{th3.2}
Let $K$ be the Sierpi\'nski gasket and $\mu$ be the normalized Hausdorff measure on $K$. Let $(\mathcal{E}^{(a,b)},\mathcal{F})$ be the Dirichlet form defined in Theorem \ref{th1.2}. Let $\Lambda_1$ be the eigenfunction space of
$\lambda_1$, the first eigenvalue of $-\Delta$ with Dirichlet boundary condition.  Then $\Lambda_1$ is of dimension one.
\end{lemma}

\begin{proof}
 We make use of the Rayleigh quotient for the first eigenvalue:
\begin{equation}\label{eq3.3}
\lambda_1=\inf\limits_{u\in\mathcal{F}_0,u\neq0}\frac{\mathcal{E}(u)}{||u||^2_2},
\end{equation}
where $\mathcal{F}_0:=\{u\in\mathcal{F}: u|_{V_0}=0\}$. There exists a function $u\in {\mathcal F}$ attains the infimum, and all such functions  must be  eigenfunctions with eigenvalue $\lambda_1$. Therefore by the Markovian property of the Dirichlet form, we see that $\Lambda_1$ is a closed {\em sublattice}, hence also $u^+$, $u^-$ are contained in $\Lambda_1$. For any $u\in\Lambda_1$, we have
\begin{equation*}
 P_t u=\sum\limits_{n=0}^\infty\frac{t^n}{n!}\Delta^nu=
 \sum\limits_{n=0}^\infty\frac{t^n}{n!}\left(-\lambda_1\right)^nu= e^{-t\lambda_1}u\leq u.
 \end{equation*}
 By using Proposition \ref{th3.1} with $\mathcal{L}=\Lambda_1$, we see that $\Lambda_1$ is of dimension at most one, and thus $\Lambda_1$ is one dimensional since $\Lambda_1$ is nontrivial.
 \end{proof}

\bigskip

\begin{lemma}\label{th3.3}
There exists $C>0$ such that for  any initial data $a>b=c>0 $ on $\Gamma_0$, we have
\begin{equation}\label{eq3.4}
C^{-1}b\leq\lambda_1^{(a,b)}\leq Cb,
\end{equation}
where $\lambda_1^{(a,b)}$ is the first eigenvalues of $-\Delta^{(a,b)}$ with the Dirichlet boundary condition.
\end{lemma}

\medskip

\begin{proof}  We will make use of the Rayleigh quotient  in \eqref{eq3.3} again.
 Let $u_1$ be the {\em1-piecewise harmonic function on $K$} with prescribed values $u_1(p_1)=u_1(p_2)=u_1(p_3)=u_1(p_{23})=0$, $u_1(p_{12})=u_1(p_{13})=1$, where $p_{ij}$ is the vertex in $V_1$ opposite to $p_k$ for distinct $i,j,k\in\{1,2,3\}$ (see Figure \ref{fig3} for the values of $u_1$).
 \begin{figure}[h]
\textrm{\centering
\scalebox{0.25}[0.25]{\includegraphics{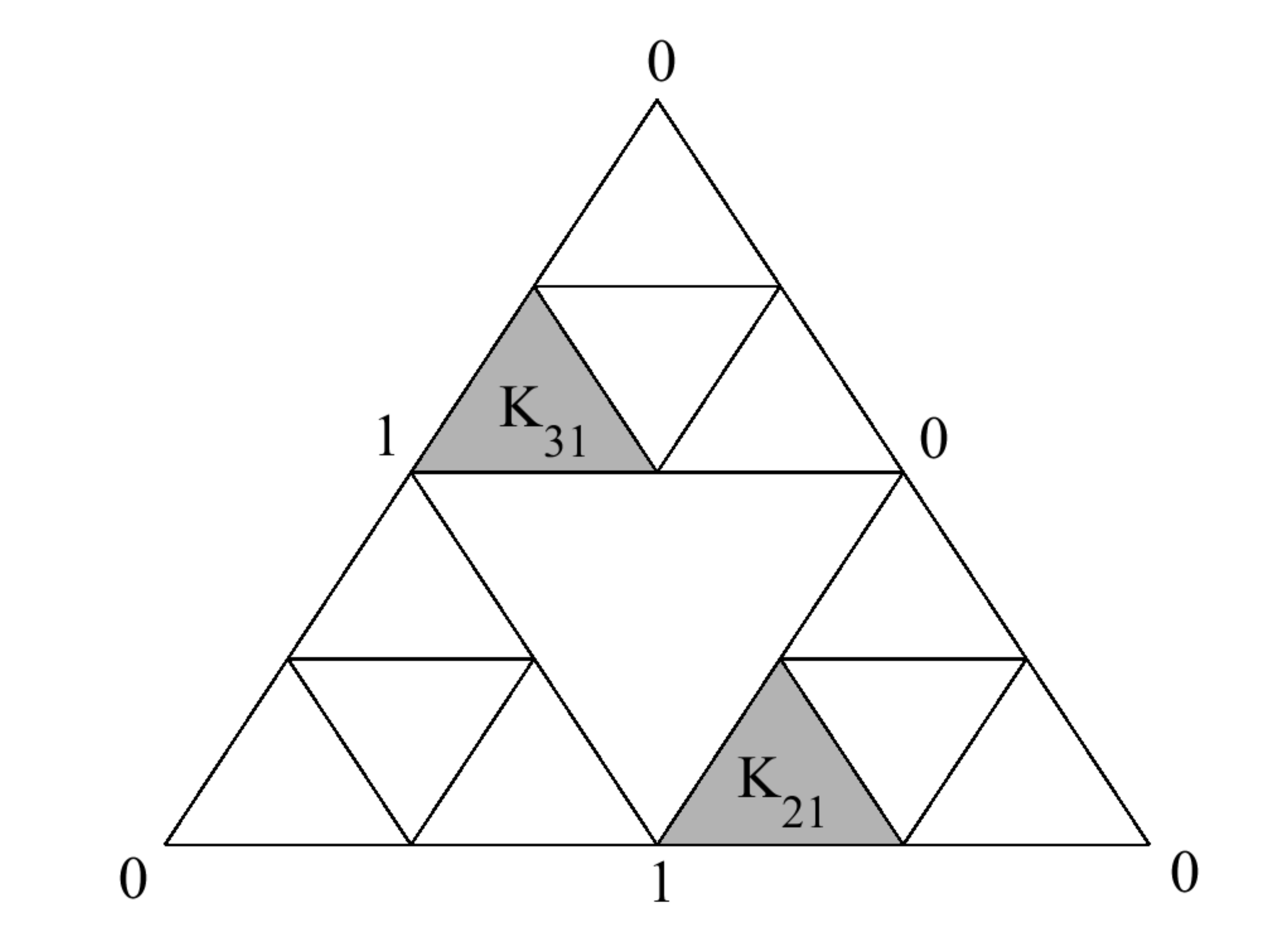}}
}
\caption{The value of $u_1$}
\label{fig3}
\end{figure}
Then by \eqref{eq2.12}
$$
||u_1||^2_2\geq\int_{F_{21}(K)\cup F_{31}(K)}u_1^2d\mu\geq  \frac29\cdot\left(\frac{a_2+b_2}{3a_2+2b_2}\right)^2 \geq \frac2{81},
$$
where $a_2, b_2$ are the second iterations of $a\ (=a_0),\  b\ (=b_0)$ respectively.
Also observe that  $\mathcal{E}^{(a,b)}(u_1)=6b_{1}$.
Therefore
$$
\lambda_1^{(a,b)}\leq\frac{\mathcal{E}^{(a,b)}(u_1)}{||u_1||^2_2}\leq C'b_{1}\leq Cb
$$
for some $C',C>0$.

\medskip

To estimate the lower bound, we  let  $u\in \mathcal{F}$,  then
$$
|u(x)-u(y)|^2\leq R^{(a,b)}(x,y)\mathcal{E}^{(a,b)}(u), \quad  x,y\in K.
$$
It follows that  for $u\in\mathcal{F}_0 = \big \{ u \in {\mathcal F}:  u|_{V_0}=0\big \}$, $u\neq 0$, by choosing $y=p_3$, we have
$$
|u(x)|^2\leq R^{(a,b)}(x,p_3)\mathcal{E}^{(a,b)}(u), \quad\forall x\in K.
$$
Integrating  both sides with respect to $\mu$, we obtain
$$
||u||_2^2\leq \int_K R^{(a,b)}(x,p_3)d\mu(x)\cdot\mathcal{E}^{(a,b)}(u).
$$
Recall that the resistance $R(x, y), \ x, y \in K$ has the expression
$
R(x,y)=\sup \big \{\frac{|u(x)-u(y)|^2}{\mathcal{E}(u)}: \ u\in\mathcal{F},\mathcal{E}(u)\neq 0 \big \}
$.
Using (\ref{eq2.11}), we have $C_1>0$ such that
$$
C_1b\leq\frac{\mathcal{E}^{(a,b)}(u)}{||u||_2^2}.
$$
Since $u$ is arbitrary, this implies that
$
 C_1b  \leq \lambda_1^{(a,b)}.
$
This completes the proof of the lemma.
\end{proof}

\bigskip

\begin{lemma}\label{th3.4}  Let $a_0 > b_0 = c_0$, then
for all $t\geq0$ and $n\geq0$,
\begin{equation}\label{eq3.5}
3^n\rho^{(a_n,b_n)}\left(\frac t{3^n}\right)\leq \rho^{(a_0,b_0)}(t), \quad \hbox {and} \quad  \rho^{(a_0,b_0)}_N(t) \leq 3^n\rho^{(a_n,b_n)}_N\left(\frac t{3^n}\right).
\end{equation}
\end{lemma}

\medskip
  Here  $\rho^{(a_n b_n)}(t)$ is the eigenvalue counting function using  $a_n> b_n= c_n$ as initial data on $V_0$. We refer to the similar proof in \cite[Propsitions 6.2, 6.3]{KL93}.
   The technique is that first we  restrict $\mathcal{E}^{(a_0,b_0)}$  on the sub-domain ${\mathcal{F}}_1:=\{u\in \mathcal{F} : u|_{V_1}=0\}$. Denote by $\rho\left(t;{\mathcal{E}}^{(a_0,b_0)},{\mathcal{F}}_1\right)$ the corresponding eigenvalue counting function, then by making use of  the identity (\ref{eq1.5}) we have the following relation
 \begin{equation*}
\rho\left(t;{\mathcal{E}}^{(a_0,b_0)},{\mathcal{F}}_1\right)=3\rho^{(a_1,b_1)}\left(\frac t3\right).
\end{equation*}
where $\frac{1}{3}$ in the bracket is the  scaling factor of $\mu$.
Using this repeatedly and that $\rho\left(t;{\mathcal{E}}_0^{(a,b)},{\mathcal{F}}_1\right)\leq\rho^{(a,b)}
(t)$, we obtain the first inequality in (\ref{eq3.5}). The second inequality can be shown by constructing another
Dirichlet form which has  domain ${\mathcal{F}}_2:=\{u: K\setminus V_1\rightarrow \mathbb{R} : u\circ F_i=f_i \text{ on $K\setminus V_0$ for some } f_i\in \mathcal{F}, i=1,2,3\}$ and using a similar argument.

\bigskip

\begin{theorem}  \label {th3.5}
Assume that $a_0>b_0=c_0$ on $\Gamma_0$, then  for $t_0 = \inf \{t:  \rho^{(a_0,b_0)}(t) >0\}$,
\begin {equation*}
 \rho^{(a_0,b_0)}(t)\ \asymp \  t^{\frac{\log3}{\log(9/2)}},\qquad  t>t_0  .
\end{equation*}
Similarly, the same inequality holds when $\rho^{(a_0,b_0)}(t)$ is replaced by $\rho^{(a_0,b_0)}_N(t)$  and  for any $t_0>0$.
\end{theorem}

\medskip

\begin{proof}
By Lemma \ref{th3.2}, we see that if we use $a_n> b_n =c_n$ as initial data on $\Gamma_0$, then we have $\rho^{(a_n,b_n)}\left(\lambda_1^{(a_n,b_n)}\right)=1$,
and
$\rho_N^{(a_n,b_n)}\left(\lambda_1^{(a_n,b_n)}\right)\leq\rho^{(a_n,b_n)}
\left(\lambda_1^{(a_n,b_n)}\right)+3=4$.
Then by Lemma \ref{th3.4}, we have
\begin{equation*}
\rho_N^{(a_0,b_0)}\left(3^n\lambda_1^{(a_n,b_n)}\right) \leq 4\cdot 3^n.
\end{equation*}
 Letting $t={3^n}\lambda_1^{(a_n,b_n)}$,  by Lemma \ref{th3.3}, we have
$t \asymp  3^n b_n \asymp 3^n (3/2)^n = (9/2)^n$ and $3^n\asymp t^{\log 3/ \log (9/2)}$. It follows that
$$
\rho^{(a_0,b_0)} (t) \leq \rho_N^{(a_0,b_0)}\left(t\right) \leq Ct^{\frac{\log3}{\log(9/2)}}
$$
for some $C>0$. The same argument yields the other inequality.
\end{proof}

\bigskip

Recall that the {\it spectral dimension} $d_s$ of a Dirichlet form is defined to be $\lim\limits_{t\rightarrow\infty}\frac{2\log\rho (t)}{\log t}$ if the limit exists. Heuristically,  $d_s/2 = d_f/ d_w$ where $d_f$ is the Hausdorff  dimension of $K$, and $d_w$ is the {\it walk dimension}  of $K$. The walk dimension is the space-time relation  $E_x (|X_t-x|^2) \approx t^{2/d_w}$ of the associate diffusion process \cite {S}, which is also the critical exponent of the  Besov space corresponding to the domain of the Dirichlet form \cite {GHL, GL, J} (see Section 4). From Theorem \ref{th3.5}, we see that for $a_0> b_0$,
 $$
 d_s = \frac  {\log 9}{\log (9/2)}, \qquad  d_w = \frac {\log 9}{\log 2} -1.
 $$
\bigskip
\bigskip

\section{\bf Other examples and remarks}
\setcounter{equation}{0}\setcounter{theorem}{0}

\bigskip

In this section, we consider two more examples. The first one is a modification of the SG such that  for $a_0 > b_0 = c_0$, the closure of $V_*$ under the resistance metric is different from the SG; the second one is detailed  in \cite {GL}, it is a p.c.f. set constructed by an IFS of $17$ maps with three boundary points,  of which the recursive construction does not yield a compatible sequence of  Dirichlet forms ${\mathcal E}_n, \ n\geq 0$.

\bigskip

Let $p_0=0, p_1=1, p_3 = \exp\left(\frac{\pi\sqrt{-1}}3\right)$.  We define the {\em twisted Sierpi\'nski gasket} \cite {MS} to be the unique nonempty compact set $K$ on $\mathbb{R}^2$ with the contractions $\{T_i\}^3_{i=1}$ such that
$
F_1(x)=\frac{\overline{x -p_1}}2\cdot \exp\left(\frac{\pi\sqrt{-1}}3\right) + p_1, \ F_2(x)=\frac{\overline{x-p_2}}2\cdot \exp\left(-\frac{\pi\sqrt{-1}}3\right)+p_2$, and  $F_3(x)=-\frac{x-p_3}2+p_3$,
(i.e., $F_i$ reflects the sub-triangle $K_i$ along the angle bisection at $p_i$).  Then the attractor $K$ is still the Sierpi\'nski gasket.  In \cite {MS}, Mihai and Strichartz investigated the self-similar energy forms on this twisted SG.

\medskip

Similar to the standard SG, by using the $\Delta$-Y transform (see Figure \ref{fig4}), the compatibility of $\{(x_n,y_n,z_n)\}_{n\geq0}$ must satisfy the following equations:
\begin{equation}\label{eq4.1}
\begin{cases}
x_{n-1}=x_n+\psi (x_n; y_n, z_n),\\
y_{n-1}=y_n+\psi (y_n; z_n, x_n),\\
z_{n-1}=z_n+\psi (z_n; x_n, y_n),
\end{cases}
\qquad  n \geq 1,
\end{equation}
where $\psi (x_n; y_n, z_n) = \frac{2y_nz_n}{x_n+y_n+z_n}$, and symmetrically for $\psi (y_n; z_n, x_n)$ and $\psi (z_n; x_n, y_n)$.
\begin{figure}[h]
\textrm{\centering
\scalebox{0.18}[0.18]{\includegraphics{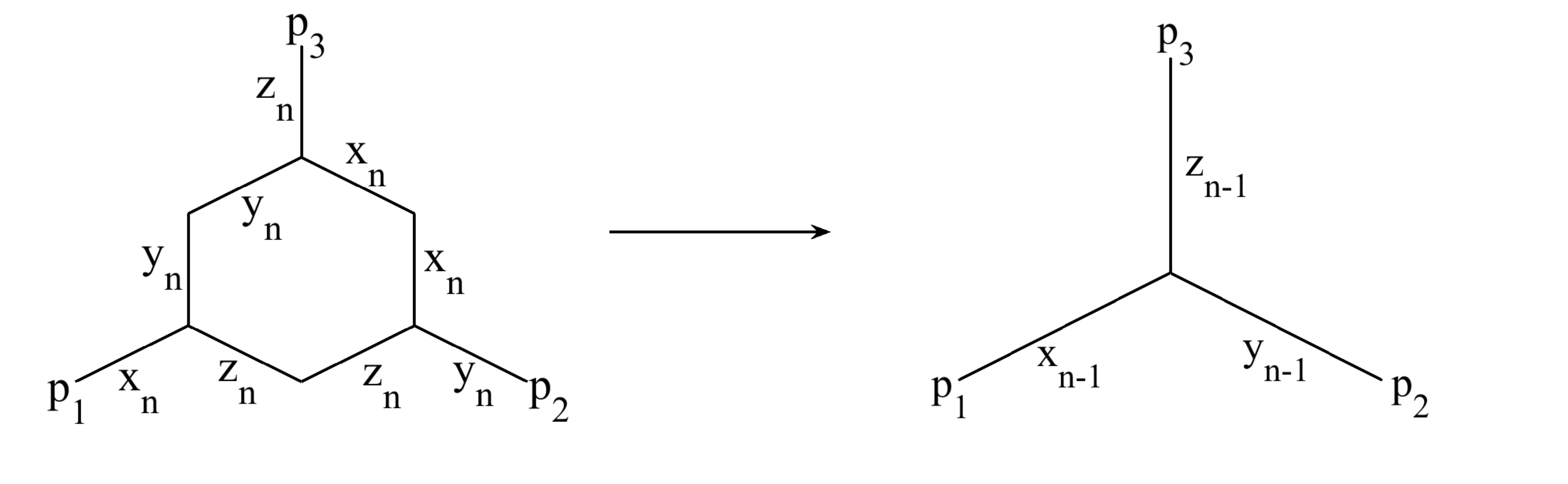}}
}
\caption{$\Delta$-Y transform for the twisted maps}
\label{fig4}
\end{figure}

\begin{lemma}\label{th4.1}
For $x_0, y_0, z_0 >0$, in order for  \eqref{eq4.1} to have positive solution $(x_n, y_n, z_n), n\geq 1$,  it is necessary and sufficient that  $ x_0 \geq y_0=z_0> 0$ (or the symmetric alternates).  In this case,  $\{(x_n, y_n, z_n)\}_{n=0}^\infty$ is uniquely determined by $(x_0, y_0, z_0)$.

\vspace {0.1cm}

Furthermore, for $x_0 > y_0 = z_0$, we  have the estimate  $x_n \asymp 1,  y_n =z_n \asymp \big ( \frac 13 \big )^n$, and hence $a_n\asymp 3^n, \ b_n =c_n \asymp 1.$
\end{lemma}

\medskip
\begin{proof} The proof for the first part is the same as Lemma \ref{th2.1}.  For the sufficiency, assuming that $x_0 \geq y_0 =z_0$, we can solve
\begin {equation} \label{eq4.2}
\begin{cases}
x_1=\frac1{10}\left(x_0+2y_0+3\sqrt{9x_0^2-4x_0y_0-4y_0^2}\right),\\
y_1(=z_1)=\frac15\left(3x_0+y_0-\sqrt{9x_0^2-4x_0y_0-4y_0^2}\right),
\end{cases}
\end {equation}
and $x_1 \geq y_1=z_1$, then proceed inductively. For the necessity, we need to change  $\rho = 2-  \frac 1{\lambda_0}$ for the estimation in \eqref{eq2.4}, and make some obvious readjustments on the calculations.

\vspace {0.1cm}
For the second part, we note that  $x_n, y_n$ can be expressed in terms of $x_{n-1}, y_{n-1}$ as in \eqref{eq4.2}. By the same estimation as in Lemma \ref{th2.2},  we have
$x_n \asymp 1,  y_n\asymp \big ( \frac 13 \big )^n$, and the estimate of $a_n, b_n$ follows.
\end{proof}

\bigskip
It follows from the estimation of the $a_n\asymp 3^n, \ b_n =c_n \asymp 1$ that
\begin{align} \label {eq4.3}
{\mathcal E}_n(u)\asymp
 {\small \ \sum_{\omega\in W_n}\Big\{3^n\Big(u_\omega(p_2)-u_\omega(p_3)\Big)^2 + \Big(u_\omega(p_1)-u_\omega(p_2)\Big)^2+\Big(u_\omega(p_1)-
 u_\omega(p_3)\Big)^2 \Big \} }
\end{align}
and the compatibility of $\{{\mathcal E}_n\}_{n\geq 0}$ implies that  for $u \in \ell(V_*)$, ${\mathcal E}(u) = \lim_{n\to \infty}{\mathcal E}_n(u|_{V_n})$ exists.

\bigskip

Let $R := R^{(a_0,b_0)}$ denote the resistance metric on $V_*\times V_*$, and  let $\Omega$ be the completion of $V_*$ with respect to $R$.   We will give a description of the topology and the completion of $(\Omega, R)$.  Let $U_0=\{p_2,p_3\}$, $U_n=\bigcup_{i=1}^3 F_i(U_{n-1})$ for $n\geq1$ and $U_*=\bigcup_{n\geq0}U_n$; also let $W_0 =\{p_1\}$, $W_n = \bigcup_{i=1}^3F_i(W_{n-1})$ for $n\geq1$ and $W_*=\bigcup_{n\geq0}W_n$ (see Figure \ref {fig5}).

\begin{figure}[h]
\textrm{\centering
\scalebox{0.2}[0.2]{\includegraphics{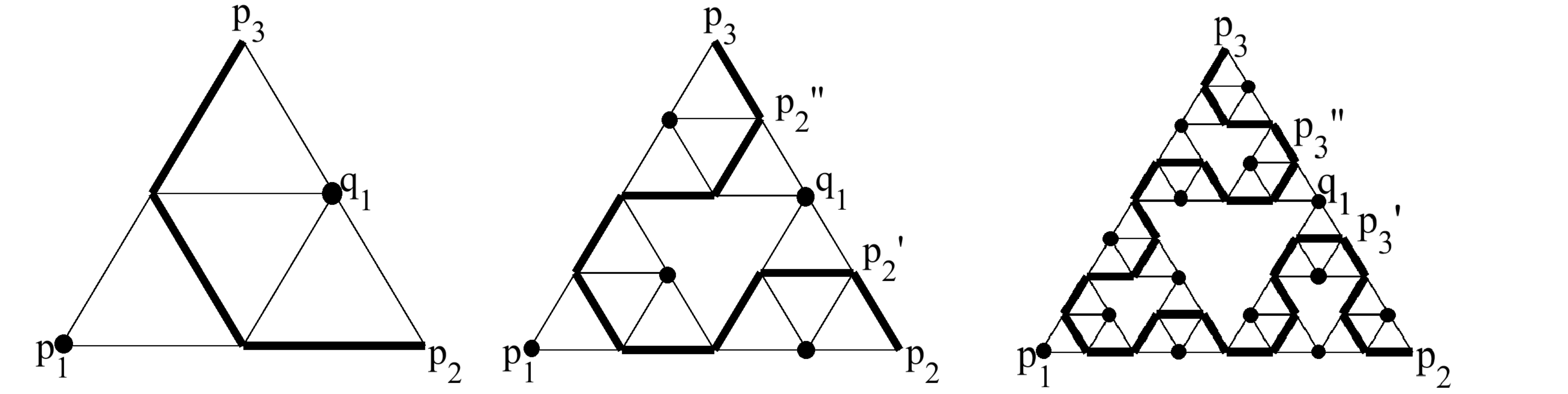}}
}
\caption{$U_n, W_n, n=1,2,3$; bold lines are the edges joining the neighboring points in $U_n$, and the bold dots forms $W_n$.}
\label{fig5}
\end{figure}

\bigskip

\begin{proposition}\label{th4.2}
On the twisted SG,

\medskip

\ (i)\ the resistance metric $R$ is a uniform discrete metric on $W_*$, and ${\rm dist}_R(U_*, W_*) >0$;

 \medskip

 (ii) On $U_*$\ , for $x\in U_*$,  let $K_x$ be the largest subcell of $K$ that has $x$ as a vertex,  then
 \begin{equation} \label{eq4.4}
 R(x, y) \asymp |x-y|^{\log3/\log2} \quad \hbox {for} \quad y \in  U_* \cap K_x;
 \end{equation}
 on the other hand,  let $q \in W_k \setminus  W_{k-1}$ with two adjacent cells $K'_q, K''_q$ (as defined above), then
 \begin{equation} \label{eq4.5}
 R(x, y) \asymp \Big ( \frac 13 \Big )^k \quad \hbox {for} \quad  x\in U_n \cap K'_q, \ y \in U_n \cap K''_q, \  n \geq k.
 \end{equation}

\vspace {0.1cm}

Consequently, the completion  $\Omega = \overline {U_*} \cup W_*$,   where  $\overline {U_*}$  is  pathwise connected and locally connected, and is such that for each $q \in W_k \setminus  W_{k-1} $, $\overline {U_*}$ has two limit points $p'_q, p''_q $ with $R(p'_q, p''_q) \asymp \Big (\frac 13\Big )^k$.
\end{proposition}

\bigskip

\noindent {\bf Remark 4.1}:  The completion $\overline {U_*} $ can be realized as cutting  up the SG at each $ q\in W_*$, and bend the two subcells $K_q', K''_q$ apart at the cut points with the appropriate distance without breaking the SG (see Figure 6 at $q_1$).

\begin{figure}[h]
\textrm{\centering
\scalebox{0.18}[0.18]{\includegraphics{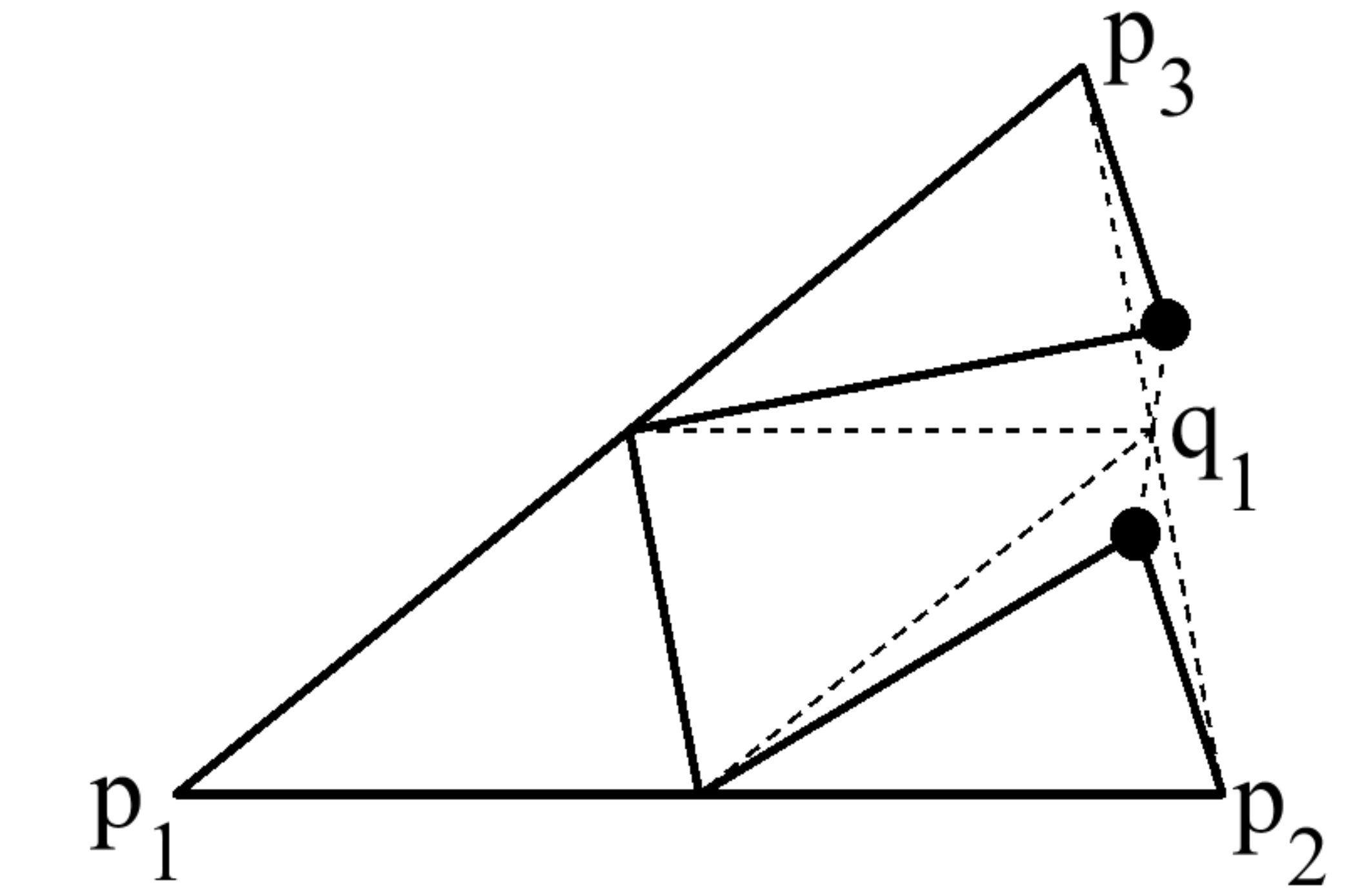}}
}

\caption{Completion of $ U_*$ at $q_1$; similarly at other $q\in W_n$ }

\label{fig6}
\end{figure}

\noindent \begin{proof} Recall that $ R(x,y)^{-1}:=\inf\{\mathcal{E}(u):\ u\in\ell (V_*), \  u(x)=1, u(y)=0\}$.
For  $x, y \in W_n$,  by using the tend functions, and observe that the effective resistance of the two nodes is $\geq b^{-1}_n/4 \asymp 1$,  we conclude  the resistance metric $R$ on $W_*$ has a uniform lower bound $>0$.  This implies $W_*$ is a uniform discrete metric space. Also by the same reason, we see that   ${\rm dist}_R (W_*, U_*) >0$.

\medskip

Next we consider the resistance metric on $U_n$. Let $K_n \subset K_x$ be the smallest subcell of $K$ that contains both $x, y$.  By using  $a_n^{-1} \asymp \Big (\frac 13\Big)^n$, and the path property of $U_n$, it is direct to prove the estimation of \eqref {eq4.4}.

To prove \eqref {eq4.5}, it suffices to consider the case  $q = q_1 \in W_1\setminus W_0$,  the midpoint of the line segment $\overline{p_2p_3}$, then use the IFS to move the argument to other $q\in W_*$. Let $p'_n \in U_n \cap F_2(K)$ that is a neighbor (in $V_n$) of $q_1$. Similarly, let  $p''_n \in U_n \cap F_3(K)$ that is a neighbor (in $V_n$) of $q_1$. Then  we have
$$
R(p'_n,  p''_n) \asymp R (p_2, p_3) \asymp 1.
$$
(For the above estimation,  note that the geodesic in $V_n$  joining $p'_n (\in F_2(K))$ and $p''_n (\in F_3(K))$ must pass through $F_1(K)$  (see Figure \ref{fig6}), so that $ R(p_2, p_3) \geq R(p'_n,  p''_n) \geq \frac 13 R(p_2, p_3)$.)

\medskip

Finally the statement of the completion $\overline U_*$  follows from \eqref{eq4.4} and \eqref{eq4.5}.  \end {proof}

\bigskip
 From the probabilistic point of view,
it will be interesting to understand  the corresponding diffusion process on $\overline {U_*}$ and $\Omega$. We also note that in \cite {HK}, Hambly and Kumagai studied this type of diffusion on the Vicsek set, they also observe that if one assigns resistance $a$ and $b$ on the side and diagonal edges with $a>b$, then the resistance between two diagonal lines on any $n$-cell has a uniform lower bound,  and thus a similar situation as the twisted SG occurs.

\bigskip

Next we list another example in \cite {GL} on which the recursive construction does not work. Let $K$ be a p.c.f. set as in Figure \ref{fig7}, which has three boundary points, and is  generated by an IFS  of 17 similitudes with contraction ratio $1/7$.  We call it  a {\it Sierpi\'nski Sickle}.
\begin{figure}[h]
\textrm{
\begin{tabular}{cc}
\begin{minipage}[t]{2.0in} \includegraphics[width=2.0in]{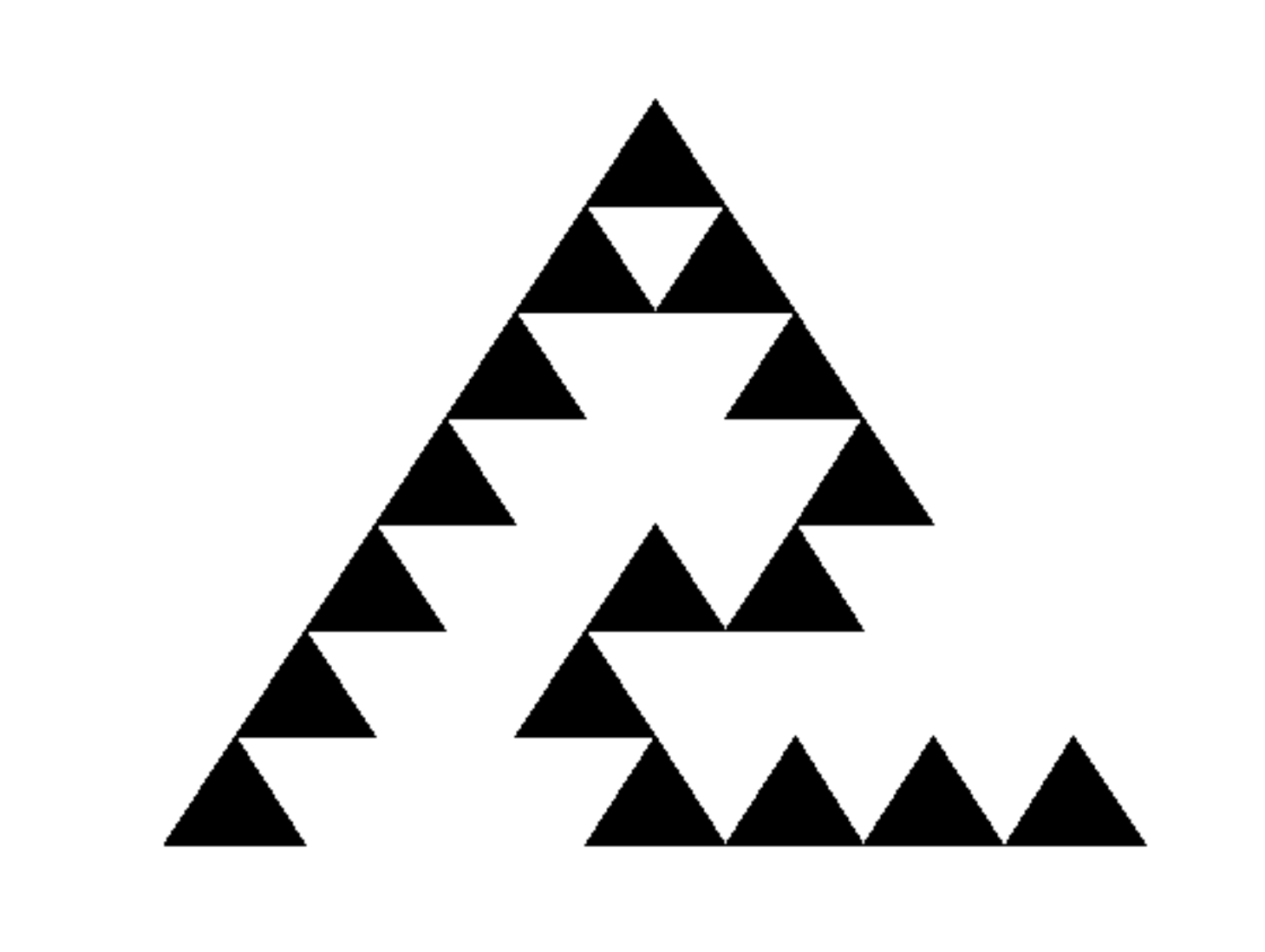}
 \end{minipage}
 \begin{minipage}[t]{2.0in}
\includegraphics[width=2.0in]{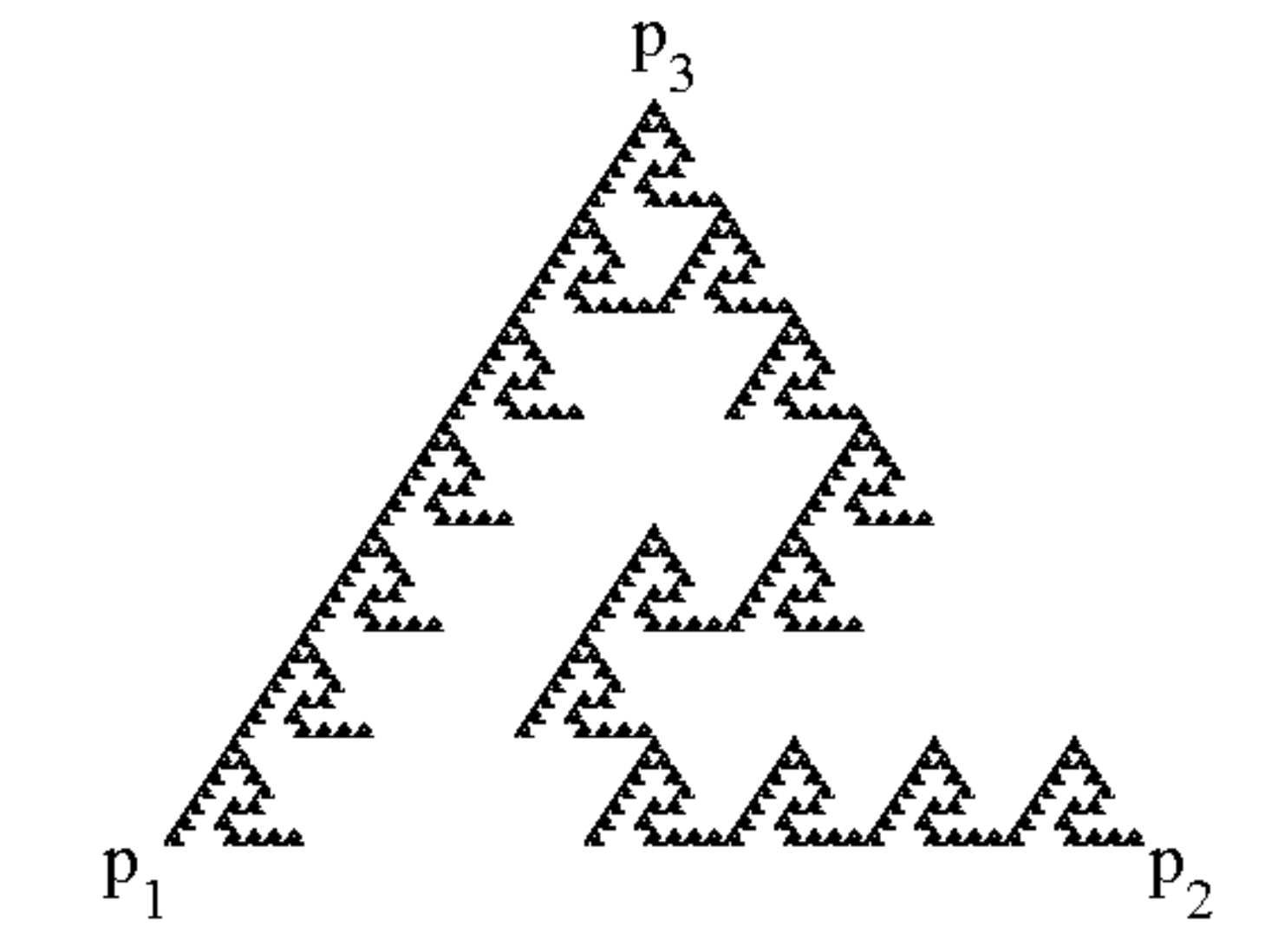}
\end{minipage} &
\end{tabular}
}\caption{The Sierpi\'nski  sickle $K$}\label{fig7}
\end{figure}

\bigskip

\begin {proposition}\label{th4.3}
 {\rm [GL]} For the Sierpi\'nski sickle, the recursive construction for any $(a_0, b_0 , c_0)$ does not give a compatible sequence of $\{(a_n, b_n, c_n)\}_{n=0}^\infty$. However, one can construct a Dirichlet form satisfying the energy self-similar identity.
\end{proposition}

\bigskip
For the recursive construction, the basic reason for no compatible sequence is that the solution similar to the system of equations \eqref{eq1.4} fails to be positive. For the self-similar case, we find the explicit renormalizing factors for $\mathcal{E}(u)= \sum_{i=1}^{17} {r_i}^{-1} \mathcal{E}(u\circ F_i)$. Let $r_L, r_R, r_T$ on the cells of $K$ be defined as follows:

\vspace{0.1cm}
$r_1=r_2=\cdots=r_5=r_{L}$\ \ on the left $5$ sub-triangles $F_1(K),F_2(K),\cdots,F_5(K)$;

\vspace {0.1cm}
$r_6=r_7=\tau_8=r_{T}$\ \  on the $3$ top sub-triangles $F_6(K),F_7(K),F_8(K)$;

\vspace {0.1cm}
$r_9=\tau_{10}=\cdots=_{17}=r_{R}$\ \  on the right $9$ sub-triangles $F_9(K),F_{10}(K),\cdots,F_{17}(K)$.

\medskip

\noindent Then  we can solve a system of equations and obtain, for $k \geq 2$,
\begin{align*}
  r_L =\frac{k(k-1)}{5(k^2+6k+3)}, \quad
  r_R =\frac{k^2-1}{(13k+5)(k^2+6k+3)}, \quad
  r_T =\frac{2(2k+1)}{k^2+6k+3}.
\end{align*}

\bigskip

\noindent {\bf Remark 4.2. } Note that  if the Dirichlet form from the recursive construction is self-similar, then all the renormalizing  factors $r_i$'s are equal. In \cite [Section 4.2] {GL}, we  have a p.c.f. set (the Vicsek eyebolted cross) that the Dirichlet forms from the recursive construction cannot satisfy the condition, hence they are all non-self-similar. Nevertheless, the self-similar Dirichlet forms can still be obtained similar to the above example.

\bigskip

\noindent {\bf Remark 4.3.}   We do not know to what extend the recursive construction and the dichotomy result can be extended to the other p.c.f. sets; also in view of the different situations on the SG and the previous examples, it will be nice to have some specific criteria on the more general fractals.

\bigskip

The original usage  of the above mentioned p.c.f sets  was to study  the critical exponents $\sigma^*$ of the Besov spaces $B^{\sigma}_{2,\infty} \subset L^2(K, \mu), \ \sigma >0$ (where $K\subset {\Bbb R}^d$ is  closed, and  $\mu$ is an $\alpha$-Ahlfors regular measure on $K$) in connection with the domain of the Dirichlet forms and the walk dimension (\cite {GL, KLW, KL}).  Recall that $B^\sigma_{2, \infty}$ has norm $||u||_{B^\sigma_{2, \infty}} = ||u||_2 + [u]_{B_{2, \infty}^\sigma}$ where
$$
[u]^2_{B_{2, \infty}^\sigma} = \sup_{0<r<1} r^{- 2\sigma} \int_K \Big (\frac 1{r^\alpha}\int_{B(x, r)} |u(x) -u(y)|^2 d\mu(y) \Big )\ d\mu(x).
$$
The critical exponent of the Besov spaces $\{B^\sigma_{2, \infty}\}_{\sigma >0}$ is defined to be
$$
\sigma^* =\sup \{\sigma:  \ B^\sigma_{2, \infty} \cap C(K) \hbox { is dense in}\ C(K)\}.
$$
For example, for the SG, $\sigma^* = \log 5/\log 2$  and $B^{\sigma^*}_{2, \infty} = {\mathcal F}$ \cite {J} (see also \cite{PP} for some nested fractals and \cite{HW} for p.c.f. fractals). In those cases,  when $\sigma > \sigma^*$, then $B^\sigma_{2, \infty}$ contents only constant functions. In \cite {GL}, we asked whether this is necessary, and investigate the relevance with the Dirichlet forms. For this we introduce another critical exponent
$$
\sigma^\# = \sup \{\sigma:  \ B^\sigma_{2, \infty}  \hbox { contains non-constant functions}\}
$$
and constructed the above example. It was shown that $ \sigma^* < \sigma^\#$ with  the explicit expressions of $\sigma^*$ and $\sigma^\#$.
Also $B^{\sigma^*}_{2, \infty} \ (\subset C(K)) $  is dense in $C(K)$, and $B^{\sigma^\#}_{2, \infty}$ is dense in $L^2(K, \mu)$ (but not dense in $C(K)$).  This Besov space  $B^{\sigma^*}_{2, \infty}$ does not support a local regular Dirichlet for $(\mathcal E,  {\mathcal F})$ with $ {\mathcal E}(u)\asymp ||u|| ^2_{B^{\sigma^*}_{2, \infty}} $ for all $u \in {\mathcal F}$.

\medskip

In all the known examples, the  critical  exponent $\sigma^*$ of the Besov spaces on $K$ equals to the walk dimension $d_w$ of $K$. This heuristic relation is not very intuitive as $\sigma^*$ is defined through the geometry of $K$, and the walk dimension is certain space-time exponent of the walk. Some of these aspects had been studied in \cite {GHL} in terms of the heat kernel.  It will be interesting to find out the more natural and direct connection of these exponents.

\bigskip

 The existence of a Dirichlet form on a fractal set still posts a fundamental and  challenging question. In \cite {Pe1}, Peirone proved there is a large class of p.c.f. self-similar sets (not necessary symmetric) possess  self-similar energy forms, and more recently, he  claimed an example of a p.c.f. set that does not admit such energy forms \cite {Pe2}.  It might be  worthwhile to  see if the recursive construction will produce a non-self-similar energy form in his example.   For the  non-p.c.f sets, it remains largely unknown for the existence of the Dirichlet forms.  Even for the Sierpin\'ski carpet, the construction is to use a probabilistic approach \cite {BB}, which is technically quite complicated, and surprisingly, there is no clear analytic approach on the discrete approximations yet.\\

\bigskip

\noindent {\bf Acknowledgement:}  The authors are indebted to  Mr. Meng Yang for many helpful discussions, in particular on the the semigroup of operators which led to Proposition 3.1 and Lemma 3.2.   They also thank Professors Hambly and Strichartz for the valuable discussions on the probabilistic counterpart, and brought their attention to the relevant references during a conference in Cornell.

\bigskip
\bigskip
\bibliographystyle{siam}

\bigskip
\end{document}